\documentclass[preprint,12pt]{elsarticle}
\usepackage{amssymb}
\usepackage{amsfonts}
\usepackage{eurosym}
\usepackage{a4wide}
\usepackage{mathrsfs}
\usepackage[T1]{fontenc}
\usepackage[utf8]{inputenc}
\usepackage{color}
\usepackage{amsmath}

\newtheorem{theo}{Theorem}

\newtheorem{prop}{Proposition}
\newtheorem{lem}{Lemma}
\newtheorem{proof}{Proof}

\newtheorem{remark}{Remark}

\begin{document}
\begin{frontmatter}
	
	

	
	\title{ Rational approximation to the Hurwitz-Lerch zeta function}


	\begin{abstract}
		
		The main result of this paper is the construction of new sequences of rational approximations to the Lerch function. This construction is based on a generalization of the "remainder Padé approximants" method introduced by the  author in 1996. More recently, this method has been applied, in the form of remainder Padé-type approximants, to the approximation of Stieltjes’ constants.
	\end{abstract}

	\author[]{M. Pr\'evost}
	\address{   Univ. Littoral C\^ote d'Opale, EA 2597- LMPA-
		Laboratoire de Math\'ematiques Pures et Apppliqu\'ees Joseph Liouville,
		62228 Calais, France, and CNRS, FR 2956, France}
 

  
 		\ead{marc.prevost@univ-littoral.fr}

 \end{frontmatter}

	\today
 	
 	
 	

\begin{abstract}
	
	The main result of this paper is the construction of new sequences of rational approximations to the Lerch function. This construction is based on a generalization of the "remainder Padé approximants" method introduced by the  author in 1996. More recently, this method has been applied, in the form of remainder Padé-type approximants, to the approximation of Stieltjes’ constants.
\end{abstract}

\section{Introduction}

For \(a\in \mathbb{C}\) with \( \Re (a)>0,\) the Hurwitz zeta function is defined by the power series 
\begin{equation*}
\zeta (s,a)=\sum_{k=0}^{\infty }\frac{1}{(k+a)^{s}},\Re (s)>1.
\end{equation*}
 It admits an   analytic continuation  to \(s\in \mathbb{C}\setminus
\{1\}\), with a single pole at \(s=1\).

	The Lerch transcendent (Hurwitz-Lerch zeta function)   is defined by   the power series
	\[
	\Phi(z, s, a) = \sum_{n=0}^{\infty} \frac{z^n}{(a + n)^s}, \quad a \neq 0, -1, -2, \dots,
	\]
	which converges for \( |z| < 1 \) for any \( s \in \mathbb{C} \) or \( |z| \leq 1 \) for \( \Re(s) > 1 \).

	When \( z =1 \), the function \( \Phi \) specializes to the classical Hurwitz zeta function:
	\[
	\Phi(z, s, a) = \zeta(s, a) := \sum_{n=0}^{\infty} \frac{1}{(n + a)^s}, \quad \Re(s) > 1,
	\]
	and in particular,
	\[
	\Phi(1, s, 1) = \zeta(s, 1) = \zeta(s) := \sum_{n=1}^{\infty} \frac{1}{n^s}, \quad \Re(s) > 1,
	\]
	which is the Riemann zeta function. 
	
	For \( a = 1 \), \( \Phi \) specializes to the polylogarithm function:
	\[
	z \Phi(z, s, 1) = z Li_s(z) := \sum_{n=1}^{\infty} \frac{z^n}{n^s}.
	\]
	
	For other values of  \( z, s, a \), the function \( \Phi(z, s, a) \) is defined by analytic continuation. In particular,
	\[
	\Phi(z, s, a) = \frac{1}{\Gamma(s)} \int_0^{\infty} x^{s-1} e^{-ax} \frac{1}{1 - z e^{-x}} \,dx, \quad \Re(a) > 0, \quad z \in \mathbb{C} \setminus [1, \infty), \quad \Re(s) > 0.
	\]

 	We define the domain \( \Omega_a \) as follows:
 \[
 \Omega_a \equiv
 \begin{cases}
 	\mathbb{C} \setminus [1, \infty) & \text{if } \Re(a) > 0, \\
 	\{ z \in \mathbb{C} : |z| < 1 \} & \text{if } \Re(a) \leq 0.
 \end{cases}
 \]
	For \( |\arg(a)| < \pi \), \( s \in \mathbb{C} \), and \( z \in \Omega_a \), the asymptotic expansion of \( \Phi(z, s, a) \) for large \( a \) is given by \cite{Ferreira}:
\begin{equation}\label{DAS Lerch}
\Phi(z, s, a) = \frac{1}{1 - z} \frac{1}{a^s} + \frac{1}{a^s}\sum_{n=1}^{N-1} \frac{(-1)^n \text{Li}_{-n}(z)}{n!} \frac{(s)_n}{a^{n}} + \mathcal{O}(a^{-N-s}),
\end{equation}
where \( (s)_n \) is the Pochhammer symbol and \( N \in \mathbb{N} \).
We may  write the following asymptotic expansion
\begin{equation}\label{lerch}
	\Phi(z, s, a)  \simeq \frac{1}{1-z}\frac{1}{a^s} + \frac{1}{a^{s}} \left( - \frac{1}{a}
	\phi_1\left(-\frac{1}{a^2}\right)+
	\phi_2\left(-\frac{1}{a^2}\right)\right)
, a  \rightarrow\infty,
\end{equation}

where \( \phi_1 \) and  \(\phi_2\) are defined
by the formal power series:

\begin{eqnarray}\label{phi1}
	\phi_1(y)&:=&\sum_{n=0}^{\infty}  \frac{  \text{Li}_{-2n-1}(z)}{(2n+1)!}  (s)_{2n+1}(-y)^{n},\\ \label{phi2} 
	\phi_2(y
	)&:=&\sum_{n=1}^{\infty}  \frac{  \text{Li}_{-2n}(z)}{(2n)!}  (s)_{2n}(-y)^{n
	}._{}
\end{eqnarray}

From the following identity
\[\Phi(z,s,a)=  \sum_{k=0}^{n-1} \frac {z^k}{(k+a)^s}+z^n \Phi(z,s,a+n),\]

we have
\begin{equation}\label{lerch}
	\Phi(z, s, a) \simeq \sum_{k=0}^{n-1} \frac {z^k}{(k+a)^s}+\frac{z^n}{a_n^{s}}\left(\frac{1}{1-z} -  \frac{1}{a_n}  \phi_1\left(-\frac{1}{a_n^2}\right)+
	\phi_2\left(-\frac{1}{a_n^2}\right)\right) 
\end{equation}

where \(a_n:=a+n\).

The central focus of the paper is to replace the asymptotic series \(
	\phi_1(-1/a_n^2)\) and \(%
	\phi_2(-1/a_n^2) \) in \eqref{lerch} with    Padé-type approximants  rather than  with  
Pad\'e approximants as in our previous papers \cite{prevostrivoaleuler,prevostrivoalDIAGRPA}. This choice is necessary   because the underlying weight functions are not positive on the positive real axis. The basic properties of Pad%
é-type approximants are recalled in Section \ref{sectionPTA}
and below \((\tau /\sigma)_f(z)\) denotes the Padé-type approximant
of \(f(z)\) with numerator and denominator of degree \(\leq \tau \) and \( \leq \sigma 
 \) respectively. The denominator is a specified polynomial of degree \( \sigma  \)
and the numerator is then determined so that the Taylor expansion of the
rational fraction matches the Taylor expansion of \( f(z) \) up to \(  z^{\tau} \) at \(  %
z=0 \). We shall consider the polynomial 

\begin{eqnarray}
T_{rn}( \alpha,\beta,x)
&:=& _3F_2\left[ 
\begin{array}{c}
	-rn,1-i\sqrt{x},1+i\sqrt{x} \\ 
	\alpha+1 ,\beta+1%
\end{array}
;1\right]\\&=&\sum_{k=0}^{rn}\frac{(-rn)_k (1-i\sqrt{x})_k(1+i\sqrt{x})_k }{(\alpha+1)_k (\beta+1)_k}\frac{1}{k!}\\
&=&\sum_{k=0}^{rn}\binom{rn}{k} \binom{i\sqrt{x}+k}{2k+1}(2k+1)!\frac{1	}{i\sqrt{x}(\alpha+1)_k(\beta+1)_k}
\end{eqnarray}

\noindent as the generating polynomial (see \S \ref{sectionPTA}) for the denominators
of the Padé-type approximants \( (rn+p/rn)_{\Psi_k},(k=1,2 )\).
 This polynomial is a special case     of Wilson polynomials \cite{wilson}; see 
\S \ref{section wilson pol} for further  details.  

In what follows, we use \(T_{rn}(\alpha,1,x)\) as the generating polynomial of the Padé-type approximant employed   in the following theorem.

\begin{theo}[Remainder Pad\'e-type approximations for \( \Phi(z,s,a) \)]
	
\label{teo1} Let \( r\in \mathbb{Q}_{\ge 0} \) be such that \( rn\in \mathbb{N}. \)
  Then for every integer \( n\geq 0\) and \( p\geq 1 \), we define \(\alpha=p+1+\lfloor \Re(s)\rfloor\).
  Then, for all \(z<0, s \in \mathbb{C}, a\in  \mathbb{R},\Re(a)>0, \)
 \begin{equation}
	\Phi(z, s, a) = \sum_{k=0}^{n-1} \frac {z^k}{(k+a)^s}+\frac{z^n}{a_n^s}\left(\frac{1}{1-z} -  \frac{1}{a_n}  (rn+p/rn)_{\phi_1}(-1/a_n^2)+ (rn+p/rn)_{\phi_2}(-1/a_n^2)+\delta_{rn,p}\right)
\end{equation}
where

 \[\limsup_{n->\infty} \left| \delta_{rn,p}\right| ^{1/n}\leq \left|z \right|/ \rho\]
 with  \(\rho:= \frac{1+rx}{(1-x)^r(1-r x)}  \)
 where  \( x \) is the solution in \( [0,1/r] \) of 
 \(r^2x^2+x-1=0\).
\end{theo}
(If \(r=1\) then \(\rho=\frac{-11+5\sqrt
5
}{2}=11.090..\))

\begin{remark}
If \(z\) and \(a\) are rational numbers and \(s\) is an integer, then the remainder Padé-type  approximant given in Theorem 1 is itself a rational number.
\end{remark}

\medskip

The paper is organized as follows. In \S \ref{sectionPTA}, we recall the
definition of Padé-type approximants and the key formula for the error. In \S \ref{proofprop1}, we gather various results needed for the proof of Theorem \ref{teo1}, in particular a moment interpretation of
the coefficients of \( \phi_1,\phi_2 \) involving certain explicit weight functions \( w_1 \) and \( w_2 \),
for which we establish an upper bound in terms of the Wilson weight. The properties of
the latter are recalled in \S \ref{section wilson pol}. Theorem \ref{teo1}
is proved in \S \ref{section proof teo1}. In \S \ref{sec:expression}, we
provide explicit expressions of the Padé-type approximants in
Theorem \ref{teo1}.  {Finally, in \S\ref{sec:numericalex}, we present numerical results to illustrate the accuracy of the bounds in Theorem \ref{teo1}.}

\section{Padé-type approximants}

\label{sectionPTA}

In this section, we recall the definition and basic properties of \emph{Pad%
é-type approximants}. See \cite[Chapter 1]{brez1980} for further 
details.

Let \( V_n  \) be a polynomial of degree \( n \) and \( f(z)=\int_0^{\infty }\frac{ \mu
(x)}{1-zx}dx \) where the weight function \( \mu(x) \) is not necessarily positive
on \( [0,+\infty) \) and is such that \( \int_0^\infty x^n \vert \mu(x)\vert
dx<\infty \) for all integer \( n\ge 0 \). \( V_ n \) is  called the \emph{%
generating polynomial} of the Padé-type approximant. Let \( %
W_{n-1} \) be the associated polynomial of \( V_n \) with respect to the weight \( %
\mu \): by definition, 
\begin{equation*}
W_{n-1}(z):=\int_0^{\infty }\frac{V_n(x)-V_n(z)}{x-z}\mu (x)dx.
\end{equation*}
and it is of degree \( \le n-1 .\)

We set \( \widetilde{W} _{n-1}(z):=z^{n-1}W_{n-1}(1/z) \) and \( \widetilde{V}_
n(z):=z^ nV_ n(1/z) \), so that 
\begin{align}
f(z)-\frac{\widetilde{W}_{n-1}(z)}{\widetilde{V}_n(z)} &=\int_0^{\infty }%
\frac{\mu (x)}{1-zx}dx-\frac{z^{-1}}{V_n(z^{-1})} \int_0^{\infty }\frac{%
V_n(x)-V_n(z^{-1})}{x-z^{-1}}\mu (x)dx  \notag \\
&=\int_0^{\infty }\frac{\mu (x)}{1-zx}dx+\int_0^{\infty }\frac{
V_n(x)/V_n(z^{-1})-1}{1-zx}\mu (x)dx  \notag \\
&=\frac{z^n}{\widetilde{V}_n(z)}\int_0^{\infty }\frac{V_n(x)}{1-zx}\mu (x)dx.
\label{eq:11512}
\end{align}
Thus, the Taylor expansion at \( z=0 \) of the rational fraction \( \widetilde{W}
_{n-1}(z)/\widetilde{\text{ }V_n}(z) \) matches   that of \( f(z) \) up to \( %
z^{n-1} \) at \( z=0 \). By analogy with Padé approximant, the
quotient \( \frac{\widetilde{W} _{n-1}(z)}{\widetilde{V}_ n(z)} \) is denoted \( %
(n-1/n)_{f}(z) \) and is called a Padé-type approximant to \( %
f(z). \) Note that if \( V_n \) is orthogonal with respect to \( \mu \) (i.e., \( %
\int_0^\infty x^kV_n(x) \mu(x)dx=0 \) for \( 0\le k\le n-1 )\), we  replace in %
\eqref{eq:11512} \( V_n  \) and \( \widetilde{V}_n \) by their square respectively,
and \( z^n \) by \( z^{2n} \) so that the error is \( O(z^{2n}). \)

More generally, we define 
\begin{equation*}
(n+p-1/n)_f:=c_0+c_{1}z+c_ 2z^ 2+\cdots +c_{p-1}z^{p-1}+z^{p}(n-1/ n)_{f_p} 
\end{equation*}
where \( p\in \mathbb{N} \) and \( f_p(z):=c_{p}+c_{p+1}z+\cdots  \). This means
that the weight \( \mu \) is  multiplied by \( x^{p} \). The error \( %
f(z)-(n+p-1/n)_{f}(z) \) then satisfies 
\begin{equation}
f(z)-(n+p-1/n)_{f}(z)=\frac{z^{p+n}}{\widetilde{V}_ n(z^{-1})}\int_0^{\infty
}\frac{ V_ n(x)}{1-xz}x^{p}\mu (x)dx.  \label{errorPTA}
\end{equation}

\section{Auxiliary results}

\label{proofprop1}

In this section, we collect and prove various results that will be used in
the proof of Theorem \ref{teo1}.

\subsection{An integral representation}

In \cite{blagouchine}, the authors proved the following formula.
For \( \Re(a) > 0 \) and \(z <0\), the Hermite-type formula for \( \Phi(z, s, a) \) is
\[
\Phi(z, s, a) = \frac{1}{2a^s}  + \int_0^\infty \frac{\sin(s \arctan(t/a) - \mu t)}{(a^2 + t^2)^{s/2} \sinh(\pi t)} \, dt.
\]
where \( \mu=\log(-z) \).
Alternatively, as shown in  \cite{prevosthurwitz}
\begin{equation}
\Phi(z, s, a) = \frac{1}{2a^s} + a^{1-s} \frac{(-1)^m}{\Gamma(s)\Gamma(m+1-s)} \int_0^\infty \frac{1}{a^2 + x^2} \left(  w_1(x) - a w_2(x) \right) dx,\label{exprintphi}
\end{equation}
where \( w_1(x) \) and \( w_2(x) \) are given by:
\[
w_1(x) = \int_x^\infty x^s(t - x)^{m-s} \frac{d^m}{dt^m} \left[  \frac{\cos(\mu t)}{\sinh(\pi t)} \right] dt,
\]
\[
w_2(x) = \int_x^\infty x^{s-1}(t - x)^{m-s} \frac{d^m}{dt^m} \left[  \frac{\sin(\mu t)}{\sinh(\pi t)} \right] dt.\]

where \( \mu=\log(-z) \) and \( m \) is an integer strictly greater than \( \lfloor{\Re(s)}\rfloor -1
\).
We set
\[\Psi_1(y):=	   \int_0^\infty \frac{-y}{1- y x^2}  w_1(x)    dx=  \int_0^\infty \frac{-y}{1- y x} \frac{ w_1(\sqrt{x})}{2 \sqrt{x}}dx ,
\]
\[\Psi_2(y):=	   \int_0^\infty \frac{-y}{1-y  x^2}  w_2(x)  dx=\int_0^\infty \frac{-y}{1- y x} \frac{ w_2(\sqrt{x})}{2\sqrt{x} }dx.
\]
So, the expression for \( \Phi  \) is
\begin{equation}\label{expression lerch}
	\Phi(z, s, a) = \frac{1}{2a^s} + a^{1-s} \frac{(-1)^m}{\Gamma(s)\Gamma(m+1-s)} \left( \Psi_1\left(-1/a^2\right)-a 
\Psi_2\left(-1/a^2\right)\right).
\end{equation}\label{12}

From the following identity
\[\Phi(z,s,a)= \sum_{k=0}^{n-1} \frac {z^k}{(k+a)^s}+z^n \Phi(z,s,a+n) ,\]

\begin{equation} 
	\Phi(z, s, a) = \sum_{k=0}^{n-1} \frac {z^k}{(k+a)^s}+z^n\left(\frac{1}{2a_n^s} + \frac{1}{a_n^{s-1} } \frac{(-1)^m}{\Gamma(s)\Gamma(m+1-s)} \left( \Psi_1(-1/a_n^2)-a_n 
	\Psi_2(-1/a_n^2)\right)\right)\label{13}
\end{equation}

where \( a_n:=a+n \).

 \textbf{Particular case}: \(s\) is an integer.
In this case, we can take \(m=s\) and the formulas simplify to
\begin{eqnarray*}
w_1(x)&=&-x^s \frac{d^m}{dx^m} \left(  \frac{\cos(\mu x)}{\sinh(\pi x)} \right)\\
 w_2(x)&=&-x^{s-1} \frac{d^m}{dx^m} \left(  \frac{\sin(\mu x)}{\sinh(\pi x)} \right) .
\end{eqnarray*}

Using formula  (\ref{13}) and the asymptotic expansions \eqref{phi1} and \eqref{phi2}, we are now in a position to obtain an integral representation for
the coefficients in \eqref{DAS Lerch}. 

\begin{prop}
\label{prop1} 	For every integer  \( k \geq 0 \),
\[
(-1)^{k}\text{Li}_{-2k-1}(z) \frac{(s)_{2k+1}}{(2k+1)!}  =  \frac{(-1)^{m+1}}{\Gamma(s)\Gamma(m+1-s)}  \int_0^{\infty} x^{2k}w_1(x) dx.
\]

For every integer \( k \geq 1 \),
\[
(-1)^k\text{Li}_{-2k}(z) \frac{(s)_{2k}}{(2k)!}  =   \frac{(-1)^{m+1}}{\Gamma(s)\Gamma(m+1-s)}   \int_0^{\infty} x^{2k}w_2(x) dx.
\]

\end{prop}
 \begin{proof}
  
We shall now compute the asymptotic expansion of the right-hand side of  Eq.(\ref{exprintphi}). We first observe that, for every integers $k, m\ge 0
$, the functions
 \[
(t,x)\mapsto x^{2k+1}\left\vert x^s(t - x)^{m-s} \frac{d^m}{dt^m} \left(  \frac{\cos(\mu t)}{\sinh(\pi t)} \right)\right\vert 
1_{\{0\le x \le t\}}(t,x)\] 
and
\[
(t,x)\mapsto x^{2k+1}\left\vert x^{s-1}(t - x)^{m-s} \frac{d^m}{dt^m} \left(  \frac{\sin(\mu t)}{\sinh(\pi t)} \right)\right\vert 
1_{\{0\le x \le t\}}(t,x)\]  are integrable on $[0,+\infty) \times [0,+\infty)$. Using the identity	$$
1-\left(-\frac{x^2}{a^2}\right)^k=\left(1+\frac{x^2}{a^2}\right) \sum_{j=0}^{k-1}\left(-\frac{x^2}{a^2}\right)^j,
$$
the following computations are licit: for any integer $k\ge0$, we have 
\begin{align*}
	\int_0^{\infty }&\frac{1}{a^ 2+x^ 2}(w_1(x)-aw_2(x))dx \\
	&= \sum_{j=0}^{k-1}\frac{(-1)^{j}}{a^{2j+2}}\int_0^{\infty }\
	x^{2j}(w_1(x)-aw_2(x))dx  +\int_0^{\infty }\frac{(-1)^ kx^{2k} }{a^{2k}(a^ 2+x^ 2)}%
	(w_1(x)-aw_2(x))dx \\
	&= \sum_{j=0}^{k-1}\frac{(-1)^{j}}{a^{2j+2}}\int_0^{\infty }\
	x^{2j}w_1(x) dx -\sum_{j=0}^{k-1}\frac{(-1)^{j}}{a^{2j+1
	}}\int_0^{\infty }\
	x^{2j}w_2(x)dx+\mathcal{O}(a^{-2k-1}),
\end{align*}
where, in the last line, $a\to +\infty$ and the implicit constant in $\mathcal{%
	O}$ depends on $k$.
Let us remark that the term \(\frac{1}{1-z} \) is equal to \(1/2+\int_0^\infty w_2(x)dx\).

Comparing this expression with Eq.(\ref{DAS Lerch}) (stated in the
introduction), unicity of such an asymptotic expansion implies that 
\begin{eqnarray}
	(-1)^{k}\text{Li}_{-2k-1}(z) \frac{(s)_{2k+1}}{(2k+1)!}  &= & \frac{(-1)^{m+1}}{\Gamma(s)\Gamma(m+1-s)}  \int_0^{\infty} x^{2k}w_1(x) dx\\
 (k\geq 1)\;\;\;	(-1)^k\text{Li}_{-2k}(z) \frac{(s)_{2k}}{(2k)!}  &= &  \frac{(-1)^{m+1}}{\Gamma(s)\Gamma(m+1-s)}   \int_0^{\infty} x^{2k}w_2(x) dx.
\end{eqnarray}
This completes the proof of Proposition \ref{prop1}.
\end{proof}

\subsection{ Bounds for the weights \( w_1(x) \) and \( w_2(x) \)}
\label{boundsweight}

The weights \( w_1 \) and \( w_2 \) are defined on the positive real axis. We can bound them by a Wilson weight and then apply Padé-type approximants for the asymptotic expansion  of the remainder.  
To choose the generating polynomial of the Padé-type
approximant, we shall first find an upper bound of the weights \( w_1,w_2  \) in terms
of a weight function for which the orthogonal polynomials are explicitly known.

 To bound the weight functions, we need  compute the 
  derivative of the function \(\left(  \frac{1	}{\sinh \pi x}\right)\).

   Set \(\nu(x):=\frac{1	}{e^x-e^{-x}}\), then 
\begin{eqnarray*}
 (-1)^m \frac{d^m\nu(x)}{dx^m}  
&=&\frac{e^x}{(e^{2x}-1)^{m+1}}\sum_{k=0}^{m}a_k^{(m)}e^{2kx}.
\end{eqnarray*}
The coefficients satisfy the following recurrence relation  \[ a_k^{(m+1)}=(2m-2k+1)a_{k-1}^{(m)}+(2k+1)a_{k}^{(m)}
\]
with initial conditions \(a_0^{(m)}=a_m^{(m)}=1.\)
Thus the coefficients are positive and we can write the inequality:

\[ (-1)^m\frac{d^m\nu(x)}{dx^m}  \leq \frac{e^{(2m+1)x}}{(e^{2x}-1)^{m+1}}\sum_{k=0}^{m}a_k^{(m)}
\]

Using the recurrence relation, we deduce that

\(\sum_{k=0}^{m+1}a_k^{(m+1)}=2(m+1)\sum_{k=0}^{m}a_k^{(m)}
\)
and thus
\(\sum_{k=0}^{m}a_k^{(m)}=2^m m!\).

We obtain the following upper bound
\begin{equation}
(-1)^m \frac{d^m}{dx^m}\left(\frac{1	}{\sinh\pi x}\right)\leq 2(2 \pi)^m m! \frac{e^{(2m+1) \pi   x}}{(e^{2\pi x}-1)^{m+1}}.\label{majder}\end{equation}

Using the derivative of a product, we get
\begin{equation}
\left\vert\frac{d^m}{dx^m}\left(\frac{\cos(\mu x)	}{\sinh\pi x}\right)\right\vert\leq 2(2 \pi)^m (1+\vert\mu\vert)^m m! \frac{e^{(2m+1) \pi   x}}{(e^{2\pi x}-1)^{m+1}},\label{majdercos}\end{equation}

and\begin{equation}
	\left\vert\frac{d^m}{dx^m}\left(\frac{\sin(\mu x)	}{\sinh\pi x}\right)\right\vert\leq 2(2 \pi)^m (1+\vert\mu\vert)^m m! \frac{e^{(2m+1) \pi   x}}{(e^{2\pi x}-1)^{m+1}}\label{majdersin}\end{equation}
where we have used the fact that the function \(\frac{e^{(2j+1)\pi x}}{(e^{2\pi x}-1)^j}\)
is increasing with respect to $j$.

\begin{lem}\label{lemma1}
	For all integer \( m ,\) and \(x>0\)
	\[  \frac{x ^{m+3}  e^{(2m+1)\pi x}}{(e^{2\pi x}-1)^{m+1}}  \leq \frac{1}{8 \pi}H\left(1+\frac{m}{2},1,x\right)\]
	where the function \( H \) is defined by 
	\[H(\alpha,\beta,x):=\left| \frac{\Gamma(\alpha+i x)\Gamma(\beta+i x)\Gamma(1+i x)}{\Gamma(2i x)}\right|^2 \]
	is a Wilson weight (see section  \ref{section wilson pol}).
\end{lem}
 
\begin{proof}
	By induction.
	If \(m=0\),\[ \frac{x ^{3}  e^{\pi x}}{(e^{2\pi x}-1)}  \leq \frac{1}{8 \pi}
	H\left(1,1,x\right).\]
	
	\(H\left(1,1,x\right)=  \frac{(\Gamma(1+i x)\Gamma(1-ix))^3}{\left| \Gamma(2ix)\right| ^2}
	=\frac{(i x\Gamma(i x)\Gamma(1-ix))^3}{\left| \Gamma(2ix)\right| ^2}
	=\frac{ \left( i x\frac{\pi}{\sin(i\pi x)} \right)^3}{\frac{\pi}{2x \sinh(2\pi x)}}
	=4x^4\pi^2\frac{\cosh(\pi x)}{\sinh^2(\pi x)}\geq 4 x^3 \pi \frac{1}{\sinh (\pi x)}=8 x^3 \pi \frac{e^{\pi x}}{e^{2\pi x}-1}
	\),
 where we have used  \(\coth(\pi x)\geq \frac{1}{\pi x}\).
	
	Now, if \( p \) is any integer, we have proved in \cite{prevostrivoalDIAGRPA}	\[\frac{H(p,1,x)}{H(p+1/2,1,x)}\frac{x}{1-e^{-2 \pi x}}=\frac{\left| \Gamma(p+i x)\right|^2 }{\left| \Gamma(p+1/2+i x)\right|^2}\frac{x}{1-e^{-2 \pi x}}\leq 1,\]
	and
	\[\frac{H(p+1/2,1,x)}{H(p+1,1,x)}\frac{x}{1-e^{-2 \pi x}}\leq 1,\]
	and the recurrence is proved.

\end{proof}

\begin{lem}\label{majw}
	Upper bound of the weight \( w_k,k=1,2\).
For all integer \(m, p, (2p+1\geq m+3)\)	\[\left|x^{m+3} w_kx)\right| \leq d_m H(2+m/2,1,x) 
	\]
	and 
	\[\left|x^{2p+1} w_k(x)\right| \leq d_m H(p+m/2+1,1,x) \]

	where the constant \(d_m=2^{m+1}
\pi^{s-1}m!
(1+\left|\mu \right| )^m \Gamma(m+1-s)\)

\end{lem}
\begin{proof}
	\begin{eqnarray*}
		\left| w_1(x)\right| &=& 	\left|\int_x^\infty x^s(t - x)^{m-s} \frac{d^m}{dt^m} \left[  \frac{\cos(\mu t)}{\sinh(\pi t)} \right] dt\right|\\&\leq & x^s \int_{x}^{\infty}(t-x)^{m-s}2(2\pi)^m m! \frac{e^{(2m+1)\pi t}}{(e^{2\pi t}-1)^{m+1}} (1+\left| \mu\right| ) ^m\\
		&\leq&2(2\pi)^m m! \frac{e^{(2m+2)\pi x}}{(e^{2x}-1)^{m+1}} (1+\left| \mu\right| ) ^m x^s
		  \int_{x}^{\infty}(t-x)^{m-s}e^{-\pi t}\\
		&\leq&2(2\pi)^m m!(1+\left|\mu \right| )^m x^s
		\frac{e^{(2m+1)\pi x}}{(e^{2\pi x}-1)^{m+1}} \pi^{s-m-1}\Gamma(m+1-s)
	\end{eqnarray*}
	
	Thus \[x^{m+3} w_1(x)\leq d_m H(2+m/2,1,x)\]
	since \( x^{s-m}\leq 1+x^2 \) and we have chosen
	\(m:=\lfloor \Re(s) \rfloor\).

	Using \(x^{2p-2}\leq (1+x^2)^{p-1}\), we can prove 
\[x^{2p+1} w_1(x)\leq d_m H(p+1+m/2,1,x).\]
\end{proof}

The proof is  similar  for \(x^{2p+1} w_2(x)\leq d_m H(p+1+m/2,1,x).\)

\section{A review of Wilson's polynomials properties}

\label{section wilson pol}

To proceed further, we make a crucial observation:
 
the function \[H(\alpha,\beta,x):=
 \left|\frac{ \Gamma(\alpha+i x)\Gamma(\beta+i x) \Gamma(1+i x) }{\Gamma(2i x)} \right|^2\] considered in Lemma \ref{lemma1} is a Wilson's weight
on \( (0, +\infty) \), for which the orthogonal polynomials are explicitly known;
see \cite{askwil1982, wilson}. We review their main properties in this section.
\\

Let \(  \alpha, \beta>0 \). Wilson's polynomials are defined by 
\begin{equation*}
P_n(\alpha ,\beta,x^2
):=( \alpha+1 ) _n
\beta+1 ) _n\; _3F_2\left[ 
\begin{array}{c}
-n,1-x,1+x \\ 
 \alpha+1 ,\beta+1%
\end{array}
;1\right] \in \mathbb{R}[x].
\end{equation*}

They satisfy the orthogonality relations 
\begin{equation*}
\int_{0 }^{+\infty }P_n( \alpha ,\beta,-x^2)P_m( \alpha ,\beta,-x^2)H(
\alpha,\beta ,x)dx=0,\quad n\neq m,
\end{equation*}
and 
\begin{equation*}
 \int_{0 }^{+\infty }P_n^2( \alpha ,\beta,-x^2)H(
\alpha,\beta ,x)dx= 2\pi n!\Gamma(\alpha+1+n)\Gamma(\beta+1+n)\Gamma(\alpha+\beta+n) .
\end{equation*}

Consider now the weight function \( \eta ( \alpha,\beta ,x):=\frac{1}{\sqrt{%
x }}H( \alpha ,\beta,\sqrt{x})\in L^{1}(\mathbb{R}^{+}) \). The sequence of
 orthogonal polynomials \( (T_ k(\alpha,\beta ,x))_ k \) on \( [0,\infty ) \)
with respect to \( \eta ( \alpha,\beta ,x) \) is given by 
\begin{equation*}
T_ n(\alpha,\beta ,x)=\frac{1}{(\alpha+1)_n(\beta+1)_n} P_{n}(\alpha ,\beta,- x)= \;_3F_2\left[ 
\begin{array}{c}
	-n,1-i\sqrt{x},1+i\sqrt{x} \\ 
	\alpha+1 ,\beta+1%
\end{array}
;1\right] \in \mathbb{R}[x].
\end{equation*}
In particular, for all integer \( n\geq 0 \),

\begin{equation}
\int_0^{+\infty }T_n^ 2( \alpha ,\beta,x)\eta ( \alpha,\beta ,x)dx= 
 4\pi n!\frac{\Gamma(\alpha+1
 	)\Gamma(\beta+1)\Gamma(\alpha+\beta+n)}{(\alpha+1)_n(\beta+1)_n} \label{eq:new3},
\end{equation}
and
\begin{equation}
	\int_0^{+\infty } \eta ( \alpha,\beta ,x)dx= 
	4\pi \Gamma(\alpha+1
	)\Gamma(\beta+1)\Gamma(\alpha+\beta)   .\label{eq:new4}
\end{equation}

\begin{eqnarray*}
T_n(\alpha,\beta,x)& := &     \;_3F_2(-n,1-i\sqrt{x},1+i\sqrt{x};\alpha+1,\beta+1;1)\\&=&\sum_{k=0}^{n}
\frac{ (-n)_k (1-i\sqrt{x})_k(1+i\sqrt{x})_k}{(\alpha+1)_k (\beta+1)_k}\frac{1}{k!}\\
&=&	\sum_{k=0}^{n}\frac{ (-n)_k (-1)^k(i\sqrt{x}-k)_{2k+1} }{k!  (\alpha+1)_k(\beta+1)_k i \sqrt{x}}\\
&=&\sum_{k=0}^{n}\binom{n}{k} \binom{i\sqrt{x}+k}{2k+1}(2k+1)!\frac{1	}{i\sqrt{x}(\alpha+1)_k(\beta+1)_k} \\&=&\sum_{k=0}^{n}\binom{n}{k} \frac{1	}{(\alpha+1)_k(\beta+1)_k}\sum_{\nu=0}^{2k+1} S_{2k+1}( \nu) \left( \sum_{j=0}^{\nu} \binom{\nu}{j} \left( i\sqrt{x} \right)^{j-1} k^{\nu-j} \right)\\
&=&\sum_{k=0}^{n}\binom{n}{k} \frac{1	}{(\alpha+1)_k(\beta+1)_k}\sum_{j=0}^{2k+1} \left( i\sqrt{x}^{j-1} k^{-j} \left( \sum_{\nu=j}^{2k+1} \binom{\nu}{j} \, S_{2k+1}( \nu) k^\nu \right) \right)\\
&=&\sum_{k=0}^{n}\binom{n}{k} \frac{1	}{(\alpha+1)_k(\beta+1)_k}\sum_{q=0}^
{k }   \left( (-1)^q x^q   \left( \sum_{\nu=2q+1}^{2k+1} \binom{\nu}{2q+1} \, S_{2k+1}( \nu) k^{\nu-2q-1} \right) \right),
\end{eqnarray*}
where the \(S_{k}( \nu)\) are  the  Stirling numbers of the first kind.
 
\section{Proof of Theorem \protect\ref{teo1}}

\label{section proof teo1}

For \( k=1,2 \), we defined  
\begin{equation*}
\Psi _ k(y)=\int_0^{ \infty }\frac{-y}{1-yx^2}w_ k(x)dx=\int_0^{ \infty }\frac{%
-y }{1-yx}\frac{w_k(\sqrt {x})}{2 \sqrt{x}}dx,\quad y\in \mathbb{C}\setminus
\lbrack 0,+\infty).
\end{equation*}
Proposition \ref{prop1} implies that, as \( y\to 0 \) in any angular open sector
that does not contain \( [0, +\infty) \), the asymptotic expansion of \( \Psi_k(z) \)
is given by \(  \phi_k(z) .\) (See \cite[\S 2]{prevostrivoaleuler} for more
details in a very similar situation).

\subsection{ A bound for the Padé-type approximants of \( \phi_1 \) and \(  \phi_2
 \)}
\label{sec:proofprop1}
In this section, we establish a  bound, choosing  as   generating polynomial of the Padé-type
approximants \( V_n(\mu,x)=T_ n(\mu ,1,x) \).

\begin{prop}
\label{theo:2p3} Let \( z\in \mathbb{C}  \) such that \( \Re (z)<0 \), and \( s>0 \). For
any integers \( p\geq [m/2-1] \) and \(  k=1,2  \), set \( \varepsilon
_{n,p,k}(z):=\Psi_k(z)-(n+p/n)_{\Phi_k}(z). \) Then, 
\begin{equation}
\left\vert \varepsilon _{n,p,k}(z)\right\vert \leq \varepsilon_{n,p}(z):= d_m \frac{\vert z\vert ^{p+1}}{\vert V_ n(\alpha
,z^{-1})\vert } \pi \frac{1}{\sqrt{n+1}}\Gamma(\alpha+1),
\label{eq:epsksz2}
\end{equation}
where \( \alpha :=p+1+m/2\)  and \( d_m:= \pi^{s-m-1}\Gamma(m+1-s)(1+\left| \alpha\right| )^m  \)
\end{prop}

\begin{proof}
The difference between the function \( \Psi _ k,k=1,2 \) and its Padé-type approximant is 
\begin{equation*}
\varepsilon _{n,p,k}(z)=\Psi _ k(z)-(n+p/n)_{\Phi _ k}(z)=\frac{z^{p+1}}{V_ n(\alpha ,z^{-1})}
\int_0^{\infty }x^{p+1}\frac{ w_ k(\sqrt{x})}{2\sqrt{x}}\frac{V_ n(\alpha
,x)}{1-zx}dx.
\end{equation*}
(See \eqref{errorPTA} in \S \ref{sectionPTA}.) Hence, using Lemma  \ref%
{majw}, the Cauchy-Schwarz inequality, \eqref{eq:new3} and %
\eqref{eq:new4}, we have that 
\begin{align*}
|\varepsilon _{n,p,k}&(z)| \\
&\le \left\vert \frac{z^{p+1}}{V_ n(\alpha ,z^{-1})}\right\vert
\int_0^{\infty }\frac{|V_ n(\alpha ,x)|}{|1-zx|} x^{p+1}|w_ k(\sqrt{x})|%
\frac{1}{2\sqrt{x}}dx\\
&\leq d_m /2\left\vert \frac{z^{p+1}}{V_ n(\alpha ,z^{-1})} \right\vert
\int_0^{\infty }{\ |V_ n(\alpha,x)|}\eta (\alpha ,1,x)dx \\
&\leq d_m/2\left\vert \frac{z^{p+1}}{V_ n(\alpha ,z^{-1})} \right\vert
\left( \int_0^{\infty }{\ |V_ n( \alpha,x)|}^ 2\eta (\alpha ,1,x)dx\right)
^{1/2}\left( \int_0^{\infty }\eta (\alpha ,1,x)dx\right) ^{1/2} \\
&\leq 2 d_m\left\vert \frac{z^{p+1}}{V_ n(\alpha,z^{-1})}\right\vert 
 \pi \frac{1}{\sqrt{n+1}}\Gamma(\alpha+1)^2.
\end{align*}
This completes the proof of Proposition \ref{theo:2p3}.
\end{proof}

\subsection{Completion of the proof of Theorem \protect\ref{teo1}}

\label{section proof teo1 bis}

In Eq. \eqref{13}, we  replace \( \Psi _ k(-1/a_n^2)(k=1,2) \) by the Pad%
é-type approximants \( (rn+p/rn)_{\Phi _ k}(-1/a_n^ 2) \) (for \( %
p\geq m/2-1 )\) considered in \S \ref{sec:proofprop1}. 

\begin{equation}
	\Phi(z, s, a) = \sum_{k=0}^{n-1} \frac {z^k}{(k+a)^s}+\frac{z^n}{a_n^s}\left(\frac{1}{1-z} - a_n   (rn+p/rn)_{\phi_1}(-1/a_n^2)+   (rn+p/rn)_{\phi_2}(-1/a_n^2)+\delta_{rn,p}\right)
\end{equation}
where 
\begin{align}
\delta_{rn,p}&=-a_n \varepsilon _{rn,p,1}(z)+\varepsilon _{rn,p,2}(z)\leq  (a_n+1)\varepsilon _{rn,p}(z)\\
\leq&2d_m (1+a_n)\left\vert \frac{z^{p+1}}{V_ {rn}(\alpha,z^{-1})}\right\vert 
\pi \Gamma(\alpha+1)^2
\frac{1 }{\sqrt{rn+1}} 
\end{align}

In a manner similar to  the previous article \cite{prevostrivoaleuler}, we obtain:

\[\limsup_{n->\infty} \left| V_{r n}(-a_n^2)\right| ^{1/n}=\frac{(1+x r)^{1+r x}}{(1-xr)^{1-x r}x^{3 r x}r^{2r x}(1-x)^{r-r x}}  \]
where \( x \) is the solution in \( [0,1/r] \) of 
\(r^2x^2+x-1=0\).

This completes the proof of Theorem \ref{teo1}.

\medskip

\section{Expression of the Padé-type approximants in Theorem 
\protect\ref{teo1}}

\label{sec:expression}

The purpose of this section is to make  the Padé-type approximants completely explicit for the functions \( \Psi _ k,( k=1,2)  \)
involved in Theorem \ref{teo1}.

For simplicity,    we define two linear functionals 
acting on the space of polynomials by 

\( \left\langle \Omega
^{(1)}(z),x^{\ell}\right\rangle :=\Omega _{2\ell+2p+1}(z),\left\langle \Omega
^{(2)}(z),x^{\ell}\right\rangle :=\Omega _{2\ell+2p+2}(z) \) where \( p\geq -1 \) is a
fixed integer, and the moments \( (\Omega _\ell^{}(z))_{\ell\in \mathbb{N}}  \) are
defined by 

\begin{eqnarray}
	\Omega_\ell(z) &:=&  
	 Li_{-\ell}(z) \frac{(s)_\ell}{\ell!}.
\end{eqnarray}

  We   recall 
\begin{equation*}
   V_{rn}( \alpha,x)=
\sum_{\ell=0}^{rn}\binom{rn}{\ell} \binom{i\sqrt{x}+\ell}{2\ell+1}(2\ell+1)!\frac{1	}{i\sqrt{x}(\ell+1)!(\alpha+1)_\ell} 
\end{equation*}
as the generating polynomial of all the denominator of the Pad\'e-type
approximants \( (rn+p/rn)_{\Phi_k}. \) The numerator of \( (rn+p/rn)_{\Phi_k} \)
depends on \( k \) and requires the computation of the associated polynomial of
degree \( rn-1  \)
\begin{equation*}
W_{rn-1}^{(k)}(t)=\left\langle \Omega ^{(k)},\frac{ {V}%
_{rn}(\alpha,x)- {V}_{rn}(\alpha,t)}{x-t}\right\rangle.
\end{equation*}
By linearity, it suffices to compute \(  \Omega ^{(k)}  \) applied to the
polynomial 

\begin{equation*}
\frac{\binom{ i\sqrt{x}+\ell}{2\ell+1}/(i\sqrt{x})-\binom{ i\sqrt{z}+\ell}{2\ell+1}/(i\sqrt{z})}{x-z}.
\end{equation*}

For this purpose, it is necessary to expand it in the canonical basis.   Hence, it is convenient to
first compute the associated expansion of the following polynomial 
 \[\binom{ i\sqrt{x}+\ell}{2\ell+1}/(i\sqrt{x}).\]
 
 We have 
\begin{eqnarray*}
 \binom{ i\sqrt{x}+\ell}{2\ell+1}/(i\sqrt{x}) &=&  \sum_{\nu=0}^{2\ell+1} \frac{S_{2\ell+1}( \nu) }{(2\ell+1)!}\left( \sum_{j=0}^{\nu} \binom{\nu}{j} \left( i\sqrt{x} \right)^{j-1} \ell^{\nu-j} \right)\\
 &=& \sum_{j=0}^{2\ell+1} \left( (i\sqrt{x})^{j-1
 } \ell^{-j} \left( \sum_{\nu=j}^{2\ell+1} \binom{\nu}{j} \, \frac{S_{2\ell+1}( \nu) }{(2\ell+1)!} \ell^\nu \right) \right) \in \mathbb{R}[x] \text{ thus}\\
 &=& \sum_{\tau=0}^
 {k }   \left( (-1)^\tau   x^\tau  \left( \sum_{\nu=2\tau+1}^{2\ell+1} \binom{\nu}{2\tau+1} \, \frac{S_{2\ell+1}( \nu) }{(2\ell+1)!} \ell^{\nu-2\tau-1} \right) \right)
 \end{eqnarray*}
 
Therefore, the associated polynomial of \(  {V}_{rn}(\alpha,x) \) with
respect to \( \Omega ^{(k)} \) is 
\begin{align*}
W_{rn-1}^{(k)}(t)&:=\left\langle \Omega ^{(k)},\frac{ { V}%
_{rn}(\alpha,x)- { V}_{rn}(\alpha,t)}{x-t}\right\rangle \\=&
\sum_{\ell=0}^{rn}\binom{rn}{\ell}  (2\ell+1)!\frac{1	}{ (\ell+1)!(\alpha+1)_\ell} 
 \left\langle \Omega ^{(k)},\frac{\binom{ i\sqrt{x}+\ell}{2\ell+1}/(i\sqrt{x})-\binom{ i\sqrt{t}+\ell}{2\ell+1}/(i\sqrt{t})}{
(x-t)}\right\rangle .
\end{align*}

    \[W_{rn-1}^{(k)}(t)= \left\langle \Omega ^{(k)},\sum_{\ell=1}^{rn} \frac{\binom{rn}{\ell}}{(\ell+1)!(\alpha+1)_\ell}  
\sum_{\tau=1}^{\ell} 
\sum_{m=0}^{\tau-1} (-1)^{\tau}x^m t^{\tau-1-m}   
\sum_{\nu=2\tau+1}^{2\ell+1} \binom{\nu}{2\tau+1} S_{2\ell+1}( \nu) \ell^{\nu-2\tau-1} \right\rangle
\]
Thus, 

\[W_{rn-1}^{(1)}(t) :=\sum_{\ell=1}^{rn} \frac{\binom{rn}{\ell}}{(\ell+1)!(1+\alpha)_\ell}  
\sum_{\tau=1}^{\ell} (-1)^{\tau} \ell^{-2\tau-1}  
\sum_{m=0}^{\tau-1} \Omega(2m+3+2p, z) (-1)^{m+p+1} t^{\tau-1-m}   
\sum_{\nu=2\tau+1}^{2\ell+1} \binom{\nu}{2\tau+1} S_1(2\ell+1, \nu) \ell^\nu. 
\]
and
\[W_{rn-1}^{(2)}(t) :=\sum_{\ell=1}^{rn} \frac{\binom{rn}{\ell}}{(\ell+1)!(1+\alpha)_\ell}  
\sum_{\tau=1}^{\ell}(-1)^{\tau}\ell^{-2\tau-1}  
\sum_{m=0}^{\tau-1} \Omega(2m+4+2p)(-1)^{m+p+1} t^{\tau-1-m}   
\sum_{\nu=2\tau+1}^{2\ell+1} \binom{\nu}{2\tau+1} S_{2\ell+1}( \nu) \ell^\nu 
\]

The Padé-type approximants \( (rn+p/rn)_{\Phi _ 1}(z) \) and \( (rn+p/rn)_{\Phi _ 2}(z)  \) are then respectively
\begin{eqnarray*}
\sum_{j=0}^{p}\Omega _{2j+1}(z)(-1)^jt^{j}+\frac{ t^{p}W_{rn-1}^{(1)}(t^{-1})}{%
{V}_{rn}(\alpha ,t^{-1})}\\
\sum_{j=1}^{p+1}\Omega _{2j}(z)(-1)^j t^{j}-\frac{ t^{p+1}W_{rn-1}^{(2)}(t^{-1})}{
	{V}_{rn}(\alpha ,t^{-1})}. 
\end{eqnarray*}

Hence,we have
\begin{equation*}
\end{equation*}
  
\begin{align*}
(rn+p/rn)_{\Phi _ 1}(-1/a_n^ 2) &= \left(\sum_{j=0}^{p}\Omega _{2j+1}(z)(-1)^jt^{j}+\frac{ t^{p}W_{rn-1}^{(1)}(t^{-1})}{%
	{V}_{rn}(\alpha ,t^{-1})}
\right)_{\vert t=-1/a_n^ 2} \\
&=\sum_{j=0}^{p}\Omega _{2j+1} (1/a_n) ^{2j}+\left( -%
\frac{1}{a_n^2}\right) ^p\frac {N_1}{D}
\end{align*}

and 

\begin{align*}
 (rn+p/rn)_{\Phi_ 2}(-1/a_n^ 2) &=\left(  \sum_{j=1}^{p+1}\Omega _{2j}(z)(-1)^j t^{j}-\frac{ t^{p+1}W_{rn-1}^{(2)}(t^{-1})}{
 	{V}_{rn}(\alpha ,t^{-1})}. \right)_{\vert t=-1/a_n^ 2}\\
 &=\left(  \sum_{j=1}^{p+1}\Omega _{2j}   (1/a_n)^{2j}-\left( -%
 \frac{1}{a_n^2}\right)^{p+1}\frac{ N_2}{%
 	D}\right)
\end{align*}

\begin{align*}
N_1 &:=  (-1)^p  
\sum_{\ell=1}^{rn} \frac{\binom{rn}{\ell}}{(\ell+1)!(1+\alpha)_\ell}  
\sum_{\tau=1}^{\ell}   \ell^{-2\tau-1} \! 
\sum_{m=0}^{\tau-1} \Omega(2m+3+2p, z)   a_n^{2(m+1-\tau)}\!\! \sum_{\nu=2\tau+1}^{2\ell+1} \binom{\nu}{2\tau+1} S_{2\ell+1} (\nu) \ell^\nu
\\N_2 &:= (-1)^p \sum_{\ell=1}^{rn} \frac{\binom{rn}{\ell}}{(\ell+1)!(1+\alpha)_\ell}  
\sum_{\tau=1}^{\ell} \ell^{-2\tau-1}  
\sum_{m=0}^{\tau-1} \Omega(2m+4+2p)  a_n^{2(m+1-\tau)}
\sum_{\nu=2\tau+1}^{2\ell+1} \binom{\nu}{2\tau+1} S_{2\ell+1}( \nu) \ell^\nu  \\
D &:=
\sum_{k=0}^{rn}\binom{rn}{k} \binom{-a_n+k}{2k+1}(2k+1)!\frac{-1	}{
	a_n(k+1)!(\alpha+1)_k} =\sum_{k=0}^{rn}\binom{rn}{k} \frac{(a_n-k)_{2k+1}}{(2k+1)!}\frac{1	}{
	a_n(k+1)!(\alpha+1)_k}.
\end{align*}
and \( \alpha:=p+1+m/2 \).

 \section{Numerical example}\label{sec:numericalex}
 We display a table of numerical results for 
 \(\phi(z,s,a)\) obtained from the RPTA formula in Theorem  \ref{teo1} for the values: \(z=-1/2,s=2,a=1,r=1,p=2\) and several values of  \(n\).
 
\[\begin{array}{|c|c|c|}
	\hline
	&\textup{RPTA}&\textup{RPTA}\\
	\hline
n&\delta_{n,1}&\left| \delta_{n,1}\right|^{1/n}\\
\hline
2&{1.66\cdot  10^{-6}}&{1.29\cdot  10^{-3}}\\
3&-4.86 \cdot 10^{-8}&3.65\cdot  10^{-3}\\
4&-1.22\cdot  10^{-10}&3.33\cdot  10^{-3}\\
5&1.58\cdot 10^{-12}&4.36\cdot 10^{-3}\\
10&-9.38\cdot  10^{-23}&6.27 \cdot 10^{-3}\\
20&4.16\cdot 10^{-39}&1.20\cdot  10^{-2}\\
30&3.49\cdot 10 ^{-54}&1.65 \cdot 10^{-2}\\ \hline
\end{array}\]
 
The values of the third column can be compared with the theoretical bound 

\[  \limsup_{n\to +\infty}{\left|\delta_{n,1}\right|^{1/n}}\leq4.5\cdot 10^{-2}
  \]
as given by Theorem \ref{teo1}.

\bibliographystyle{acm}

\end{document}